\documentclass{article}

\usepackage{amsmath,amssymb,amsthm}
\usepackage{hyperref}
\usepackage{graphics}
\usepackage{color}
\usepackage{ulem}
\newtheorem{Fact}{Fact}[section]
\newtheorem{Theorem}[Fact]{Theorem}
\newtheorem{Definition}[Fact]{Definition}
\newtheorem{Proposition}[Fact]{Proposition}
\newtheorem{Lemma}[Fact]{Lemma}
\newtheorem{Corollary}[Fact]{Corollary}
\newtheorem{Remark}[Fact]{Remark}

\theoremstyle{definition}
\newtheorem*{Acknowledgment}{Acknowledgment}

\title{Some positivity results of the curvature on the group corresponding to the incompressible Euler equation with Coriolis force}

\author{Taito Tauchi\thanks{
Institute of Mathematics for Industry, Kyushu University, Nishi-ku, Fukuoka, 819-0395, Japan, E-mail address: tauchi.taito.342@m.kyushu-u.ac.jp}
\and 
Tsuyoshi Yoneda\thanks{
Graduate School of Mathematical Sciences, University of Tokyo, Komaba 3-8-1 Meguro, 
Tokyo 153-8914, Japan
E-mail address: yoneda@ms.u-tokyo.ac.jp}}

\begin{document}
	\maketitle
	\begin{abstract}
	In this article,
we investigate the geometry of 
a central extension 
$\widehat{\mathcal D}_{\mu}(S^{2})$
of 
the group of volume-preserving diffeomorphisms of the 2-sphere
equipped with the $L^{2}$-metric,
whose geodesics correspond 
solutions of the incompressible 
Euler equation with Coriolis force.
In particular,
we calculate the {\it Misio{\l}ek curvature} of this group.
This value is related to the existence of a conjugate point
and
its positivity 
directly implies the positivity of the sectional curvature.
	\end{abstract}
	{\bf Keywords}: inviscid fluid flow, diffeomorphism group, conjugate point, Coriolis force, curvature, central extension.\\
{\bf MSC2020;} Primary 35Q35; Secondary 35Q31.
	\section{Introduction}
	The incompressible Euler equation on a Riemannian manifold $M$ is given by 
	\begin{eqnarray}
	 \frac{\partial u}{\partial t}+\nabla_{u}u&=&-\operatorname{grad}p,
	 \nonumber\\
	 \operatorname{div}u&=& 0,
	 \label{E-eq}\\
	 u|_{t=0}&=&u_{0}.
	 \nonumber
	\end{eqnarray}
	Its solutions correspond to geodesics on the group ${\mathcal D}_{\mu}(M)$ of volume-preserving diffeomorphisms of $M$ with $L^{2}$-metric $\langle,\rangle$, which was discovered by V. I. Arnol'd \cite{A}.
	In the case of $M={\mathbb T}^{2}$, the flat torus,
	G. Misio{\l}ek calculated the second variation of a geodesic corresponding to a certain stationary solution $X$ of \eqref{E-eq}
	and showed that the existence of a conjugate point along it.
	Moreover, he also
	revealed the importance of the value 
	\begin{eqnarray*}
	MC^{\mathfrak g}_{X,Y}&:=&
		-||[X,Y]||^{2}
	-\langle X,[[X,Y],Y]\rangle
	\end{eqnarray*}
	where $Y \in {\mathfrak g}$ and ${\mathfrak g}$ is the space of divergence-free vector fields.
	Namely, he essentially proved Fact \ref{Mcriterion},
	which states that the $MC^{\mathfrak g}_{X,Y}>0$ ensures the existence of a conjugate point on the group.
	We call this important value $MC^{\mathfrak g}_{X,Y}$
	the {\it Misio{\l}ek-curvature} 
	and want to study when it is positive or negative.
	We note that the existence of conjugate points along a geodesic
	is related to some stability of corresponding solution in this context.
	
	For $s\geq 1$, define a 2-dimensional manifold $M_{s}$ by
	$$
	M_{s}:=\{(x,y,z)\in {\mathbb R}^{3}\mid
	x^{2}+y^{2}=s^{2}(1-z^{2})\}.
	$$
	Note that
	$M_{s}=S^{2}$ if $s=1$.
	In this article,
	we calculate the Misio{\l}ek curvature 
	in the case of the incompressible Euler equation 
	with Coriolis force $az$ $(a>0)$ on $M_{s}$:
	\begin{eqnarray}
	 \frac{\partial u}{\partial t}+\nabla_{u}u&=&az\star(u)-\operatorname{grad}p,
	 \nonumber\\
	 \operatorname{div}u&=& 0,
	 \label{E-eq-C}\\
	 u|_{t=0}&=&u_{0},
	 \nonumber
	\end{eqnarray}
	where $\star$ is the Hodge operator.
	(This equation is sometimes used as the model of flows on the earth,
	see \cite{BA}.)
	In this case,
	solutions correspond to geodesics on the central extension
	$\widehat{\mathcal D}_{\mu}(M_{s})$ of the group of volume-preserving diffeomorphisms of $M_{s}$ (see Section \ref{CE}),
	whose tangent space at the identity is identified with ${\mathfrak g}\oplus {\mathbb R}$.
	Our main result is the Misio{\l}ek curvature $MC^{{\mathfrak g}\oplus {\mathbb R}}$ of this group in the direction to a west-facing zonal flow:
	\begin{Definition}
	We call a vector field $Z$ on $M_{s}$ a {\it zonal flow}
	if $Z$ has the form
	\begin{eqnarray*}
		Z=F(z)(x\partial_{y}-y\partial_{x})
	\end{eqnarray*}
	for some function $F$.
	Moreover, if $F\leq 0$,
	 we call $Z$ a {\it west-facing} zonal flow.
	\end{Definition}
	Note that any zonal flow is a stationary solution of 
	\eqref{E-eq} and \eqref{E-eq-C}.
	Then, our main results are the following:
\begin{Theorem}
\label{Main}
Let $Z$ be a nonzero west-facing zonal flow and $a\in {\mathbb R}_{>0}$.
Then we have
$$
MC^{{\mathfrak g}\oplus{\mathbb R}}_{(Z,a),(Y,b)}> MC^{\mathfrak g}_{Z,Y}
$$
for any $(Y,b)\in{\mathfrak g}\oplus {\mathbb R}$.
\end{Theorem}
Note that the definition of the west-facing zonal flow is just for simplicity.
We can easily generalize it.

\begin{Corollary}
\label{>0}
Suppose $s>1$.
	Let $Z$ be a nonzero west-facing zonal flow 
	whose support is contained in 
	$M_{s}\backslash \{(0,0,\pm 1)\}$
	and $a\in{\mathbb R}_{>0}$.
	Then,
	there exists $Y\in {\mathfrak g}$ satisfying
	 $MC_{(Z,a),(Y,b)}>0$
	 for any $b\in{\mathbb R}$.
\end{Corollary}

The theorem states that
for any west-facing zonal flow $Z$,
the Misio{\l}ek curvature of $Z$ regarded as a solution of \eqref{E-eq-C} is grater than
the Misio{\l}ek curvature regarded as a solution of \eqref{E-eq}.
We note that the positivity of the Misio{\l}ek curvature directly implies the positivity of the sectional curvature on the corresponding group (see Definition \ref{Appx-MC-def} and Lemma \ref{MC-other} in Appendix).
Thus,
the corollary implies the positivity of the sectional curvature on $\widehat{\mathcal D}_{\mu}(M_{s})$ under some support condition.

Our motivation of this study is the existence of a stable multiple zonal jet flow on Jupiter
whose mechanism is not yet well understood.
See \cite[Section.\:1]{TTconj} and references therein for more explanations and related studies.

\begin{Acknowledgment}
	The authors are very grateful to G. Misio{\l}ek for the very fruitful discussion. 
	Research of TT was partially supported by Grant-in-Aid for JSPS Fellows (20J00101), Japan Society
for the Promotion of Science (JSPS).
Research of TY  was partially supported by Grant-in-Aid for Scientific
Research B (17H02860, 18H01136, 18H01135 and 20H01819), Japan Society
for the Promotion of Science (JSPS).
\end{Acknowledgment}

\section{Misio{\l}ek curvature}
In this section,
we define the Misio{\l}ek curvature and explain its importance.
We refer to \cite{Misiolek-T,TTconj}.

Let $G$ be a (infinite-dimensional) Lie group with right-invariant metric $\langle,\rangle$,
and ${\mathfrak g}$ the Lie algebra of $G$.
Then, we define the Misio{\l}ek curvature $MC_{X,Y}:=MC^{\mathfrak g}_{X,Y}$ by
\begin{eqnarray}
MC_{X,Y}:=-||[X,Y]||^{2}-\langle X,[[X,Y],Y]\rangle.
\label{MC-form-2}
\end{eqnarray}
The first importance of this value is that the positivity of $MC$ directly implies that of the curvature (see Definition \ref{Appx-MC-def} and Lemma \ref{MC-other} in Appendix).
Note that this formula of $MC$ seems to be simpler than the general formula of the curvature on the group with right-invariant metric (see Lemma \ref{curvature-R-inv}).
The second and main importance of $MC$ is Fact \ref{Mcriterion} given below.
In \cite{Misiolek-T},
this fact is proved for the case that $G$ is the group ${\mathcal D}^{s}_{\mu}(T^{2})$ of volume-preserving $H^{s}$-diffeomorphisms of the 2-dimensional flat torus $T^{2}$.
(For the case ${\mathcal D}^{s}_{\mu}(M)$,
where $M$ is a compact $n$-dimensional Riemannian manifold,
see also \cite{TTconj}.)
The essential point of the proof in \cite{Misiolek-T} is the fact 
that the inverse function theorem holds for the 
Riemannian exponential map $\exp:T_{e}{\mathcal D}^{s}_{\mu}(M)
\to {\mathcal D}^{s}_{\mu}(M)$.
Here,
we say that
the inverse function theorem holds for $\exp$
if
$\exp$ is isometry near $X\in T_{e}{\mathcal D}^{s}_{\mu}(M)$
whenever
 the differential of $\exp$ is an isomorphism at $X$.
Thus,
we obtain the following:

\begin{Fact}
	\label{Mcriterion}
	Suppose that there exists the (Riemannian) exponential map $\operatorname{exp}:{\mathfrak g}\to G$ and the inverse function theorem holds for $\exp$. 
	Let 
	$X\in{\mathfrak g}$
	be a stationary solution of the Euler-Arnol'd equation. 
	Suppose that there exists  
	$Y\in{\mathfrak g}$ satisfying $MC_{X,Y}>0$.
Then,
there exists
a point conjugate
to the identity element $e\in G$ along the geodesic corresponding to $X$ on $G$.
\end{Fact}

Fact \ref{Mcriterion} states that the positivity of the Misio{\l}ek curvature ensures that 
the existence of a conjugate point.

\section{Central extension of volume-preserving diffeomorphism group}
\label{CE}
In this section,
we briefly recall about the central extension of the volume-preserving diffeomorphism group
by a Lichnerowicz cocycle.
Our main references are \cite{Mstability,Vizman-GEDG}.

Let $(M,g)$ be a compact $n$-dimensional Riemannian manifold
and ${\mathcal D}_{\mu}(M)$
the group of volume-preserving $C^{\infty}$-diffeomorphisms
with the $L^{2}$-metric:
\begin{eqnarray}
	\langle X,Y\rangle:=
	\int_{M}g(X,Y)\mu,
	\label{L^{2}}
\end{eqnarray}
	where $\mu$ is the volume form.
	We write ${\mathfrak g}$ for the space of divergence-free vector fields on $M$,
	which is identified with the tangent space of ${\mathcal D}_{\mu}(M)$ at the identity element.

For a closed 2-form $\eta$,
we define a Lichnerowicz 2-cocycle $\Omega$ on ${\mathfrak g}$ by
\begin{eqnarray}
	\Omega(X,Y):=
	\int_{M}\eta(X,Y).
\end{eqnarray}
	If $\eta\in H^{2}(M,{\mathbb Z})\subset H^{2}(M,{\mathbb R})$, 
	this cocycle integrates the group ${\mathcal D}_{\mu}^{\rm ex}(M)$ of 
	exact volume preserving diffeomorphism.
	This group coincides with
	 the identity component of
	${\mathcal D}_{\mu}(M)$
	if $H^{n-1}(M,{\mathbb R})=0$.
	Thus, in such a case,
	there exists a central extension $\widehat{\mathcal D}_{\mu}(M)$ of the identity component of ${\mathcal D}_{\mu}(M)$, whose tangent space at the identity is
	${\mathfrak g}\oplus {\mathbb R}$
	and its Lie bracket and inner product are given by
	\begin{eqnarray}
	[(X,a),(Y,b)]&=&
	([X,Y],\Omega(X,Y)),
	\label{CEL}
	\\
	\langle (X,a),(Y,b) \rangle
	&=&
	\langle X,Y\rangle
	+
	ab.
	\label{CEI}
	\end{eqnarray}
	Take a $(n-2)$-form $B$
	satisfying $\eta=\iota_{B}(\mu)$, or equivalently,
	\begin{eqnarray*}
		\Omega(X,Y)
		=
		\int_{M}\eta(X,Y)\mu
		=
		\int_{M}\mu(B,X,Y)
		=
		\int_{M}g(B\times X,Y),
	\end{eqnarray*}
	where $B\times X:=\star(B\wedge X)$.
	Note that this can be rewritten as 
	\begin{eqnarray}
	\label{omega-P}
		\Omega(X,Y)
		=
		\langle P(B\times X), Y\rangle,
	\end{eqnarray}
	where $P:{\mathfrak X}(M)\to {\mathfrak g}$ is the projection to the divergence-free part.
Then, the Euler-Arnol'd equation of $\widehat{\mathcal D}_{\mu}(M)$ is
\begin{eqnarray}
	\frac{\partial u}{\partial t}
	=
	-\nabla_{u}u
	+au\times B
	-
	\operatorname{grad}p.
	\label{A-E-eq-D}
\end{eqnarray}
\begin{Remark}
This formula is slightly different from \cite{Vizman-GEDG}
because the sign convention differs.
In order to clarify this,
we summarize our conventions in Appendix.
\end{Remark} 
	We summarize the contents of this section in $\dim M=2$.
	\begin{Proposition}
	\label{CEn=2}
	Suppose that $\dim M=2$,
	$\eta\in H^{2}(M,{\mathbb Z})$
	and
	$H^{1}(M,{\mathbb R})=0$.
	Then there exists a central extension group $\widehat{\mathcal D}_{\mu}(M)$ of the identity component of ${\mathcal D}_{\mu}(M)$, whose Euler-Arnol'd equation is
		\begin{eqnarray}
	\frac{\partial u}{\partial t}
	=
	-\nabla_{u}u
	+au\star B
	-
	\operatorname{grad}p.
	\label{E-A-2}
\end{eqnarray}	
	Moreover,
	the Misio{\l}ek curvature $MC_{(X,a),(Y,b)}:=MC_{(X,a),(Y,b)}^{{\mathfrak g}\oplus {\mathbb R}}$
	is given by
	\begin{eqnarray}
		MC_{(X,a),(Y,b)}&=&
		-||[X,Y]||^{2}
		-\langle X,[[X,Y],Y]\rangle
		-\Omega(X,Y)^{2}
		-a\Omega([X,Y],Y)
		\nonumber\\
		&=&
		MC_{X,Y}
		-\Omega(X,Y)^{2}
		-a\Omega([X,Y],Y)
		\label{MC-CE}
	\end{eqnarray}
	for $(X,a),(Y,b)\in {\mathfrak g}\oplus {\mathbb R}$.
	\end{Proposition}
\begin{proof}
Note that $u\times B=u\star B$ by $\dim M=2$.
	Moreover, the assertion of the Misio{\l}ek curvature follows from the definition, \eqref{CEL} and \eqref{CEI}.
\end{proof}

\section{$M=M_{s}$ case}
In this section,
we apply the results in Section \ref{CE}
to the case $M=M_{s}$.
Recall 
$$
	M_{s}:=\{(x,y,z)\in {\mathbb R}^{3}\mid
	x^{2}+y^{2}=s^{2}(1-z^{2})\}.
	$$
\begin{Proposition}
\label{Existence-D}
Let $B:=z$ and $\eta:=z\mu$.
	Then, there exists a group $\widehat{\mathcal D}_{\mu}(M_{r})$ whose Euler-Arnol'd equation is \eqref{E-A-2}.
\end{Proposition}	
\begin{proof}
	Note that $\eta =0\in H^{2}(M_{s},{\mathbb Z})$ because
	$$
	\int_{M_{s}}z\mu=0.
	$$
	Thus, the proposition follows from Proposition \ref{CEn=2}.
\end{proof}
	
Take a ``spherical coordinate'' of $M_{s}$:
\begin{equation*}
\begin{array}{cccccc}
\phi:=\phi_{s}
&
:
&
(-d,d)\times (-\pi,\pi)
&
\to 
&
M_{s}
\\
&&
(r,\theta)
&
\mapsto
&
(c_{1}(r)\cos\theta,
c_{1}(r)\sin\theta,
c_{2}(r)
)
&
\end{array}
\end{equation*}
in such a way that 
$c_{2}(0)=0$,
$c_{1}(r)>0$, $\dot{c}_{2}(r)>0$,
and that
$\dot{c}_{1}^{2}+\dot{c}_{2}^{2}=1$.
Note that
$(c_{1},c_{2},d)=(\cos(r),\sin(r),\pi/2)$
in the case of $s=1$ ($M_{1}=S^{2}$).
Then,
we obtain
$$
g(\partial_{r},\partial_{r})=
1,
\quad
g(\partial_{r},\partial_{\theta})=0,
\quad
g(\partial_{\theta},\partial_{\theta})=c_{1}^{2}
$$
and
$$
\mu=c_{1}(r)d\theta\wedge dr.
$$
This implies
\begin{equation*}
\star\partial_{r}=\frac{-\partial_{\theta}}{c_{1}},
\quad
\star dr=-c_{1}d\theta,
\qquad
\star \partial_{\theta}=c_{1}\partial_{r},
\quad
\star d\theta=\frac{dr}{c_{1}},
\label{formula-Hodge-RSCS}
\end{equation*}
and
\begin{equation*}
\langle X,Y\rangle
=
\int_{-d}^{d}
\int_{-\pi}^{\pi}
\left(
X_{1}Y_{1}
+
X_{2}Y_{2}
c_{1}^{2}
\right)c_{1}
d\theta dr
\label{integration}
\end{equation*}
for
$X=
X_{1}\partial_{r}
+
X_{2}\partial_{\theta}$
and
$Y=
Y_{1}\partial_{r}
+
Y_{2}\partial_{\theta}$,
which are elements of ${\mathfrak g}$.
Moreover,
we have
\begin{eqnarray*}
\operatorname{grad}f&=&\partial_{r}f\partial_{r}+c_{1}^{-2}\partial_{\theta}f\partial_{\theta},\\
\operatorname{div}u&=&(\partial_{r}+c_{1}^{-1}\partial_{r}c_{1})u_{1}
                    +\partial_{\theta}u_{2}
\end{eqnarray*}         
for 
a function $f$ on $M$
and
$u=u_{1}\partial_{r}+u_{2}\partial_{\theta}$.
Recall that we call a vector field $Z$ on $M_{s}$ a zonal flow if 
$Z$ has the form
$$
Z = F(r)\partial_{\theta}
$$
for some function, which depends only on the variable $r$.
Moreover,
if $F\leq 0$,
we call $Z$ a west-facing zonal flow.

\begin{Lemma}
\label{omegaZ}
	Let $Z=F(r)\partial_{\theta}$ be a zonal flow.
	Then, we have
	\begin{eqnarray*}
	\Omega(Z,Y)
	&=&
	0
	\\
		\Omega(Y,[Y,Z])&=&\int_{-d}^{d}\int_{-\pi}^{\pi}
	c_{1}^{2}
	Y_{1}^{2}
	F
	\partial_{r}c_{2}
	dr d\theta.
	\end{eqnarray*}
for $Y=Y_{1}\partial_{r}+Y_{2}\partial_{\theta}\in{\mathfrak g}$.
\end{Lemma}

\begin{proof}

Recall that $B=z=c_{2}(r)$
and 
that 
$$
\Omega(Z,Y)=\langle P(B\times Z), Y\rangle
=
\langle P(B\star Z),Y \rangle.
$$
The last equality follows from $\dim M=2$. On the other hand,
\begin{eqnarray*}
B\star Z =c_{1}(r)c_{2}(r) F(r)\partial_{r}.
\end{eqnarray*}
This expression implies that there exists a function $f$ satisfying
$\operatorname{grad}f=B\star Z $.
Thus we have $P(B\star Z)=0$
which implies the first equality.

For the second equality,
we have
\begin{eqnarray*}
	[Y,Z]
	&=&
	-F\partial_{\theta}Y_{1}\partial_{r}
	+
	\left(
	Y_{1}\partial_{r}F
	-
	F\partial_{\theta}Y_{2}
	\right)
	\partial_{\theta}
	\\
	\star BY
	&=&
	B
	\left(
	c_{1}Y_{2}\partial_{r}
	-\frac{Y_{1}}{c_{1}}\partial_{\theta}
	\right).
\end{eqnarray*}
Moreover,
\begin{eqnarray*}
	&&\Omega(Y,[Y,Z])
	\\
	&=&
	\langle
	P(\star BY),
	[Y,Z]
	\rangle
	\\
	&=&
	\langle
	\star BY,
	[Y,Z]
	\rangle
	\\
	&=&
	\int_{-d}^{d}
	\int_{-\pi}^{\pi}
	B
	\left(
	-Fc_{1}Y_{2}
	\partial_{\theta}Y_{1} 
	-
	c_{1}Y_{1}
	\left(
	Y_{1}\partial_{r}F
	-
	F\partial_{\theta}Y_{2}
	\right)
	\right)
	c_{1}drd\theta
	\\
	&=&
	-
	\int_{-d}^{d}
	\int_{-\pi}^{\pi}
	B
	\left(
	Fc_{1}^{2}(Y_{2}\partial_{\theta}Y_{1} 
	-
    Y_{1}\partial_{\theta}Y_{2})
	+
	c_{1}^{2} Y_{1}^{2} \partial_{r}F 
	\right)
	drd\theta.
\end{eqnarray*}
This is equal to
\begin{eqnarray*}
	&=&
	-
	\int_{-d}^{d}
	B
	Fc_{1}^{2}
	\int_{-\pi}^{\pi}
	\partial_{\theta}(Y_{1}Y_{2})
	d\theta dr
	\\
	&&
	-
	\int_{-d}^{d}
	\int_{-\pi}^{\pi}
	B
	\left(
	c_{1}^{2} Y_{1}^{2} \partial_{r}F
	-
	2c_{1}^{2}Y_{1} F\partial_{\theta}Y_{2} 
	\right)
	drd\theta
	\\
	&=&
	-
	\int_{-d}^{d}
	\int_{-\pi}^{\pi}
	B
	\left(
		c_{1}^{2} Y_{1}^{2} \partial_{r}F
	-
	2c_{1}^{2}Y_{1} F\partial_{\theta}Y_{2}
	\right)
	drd\theta.
\end{eqnarray*}
Recall that
$$
\operatorname{div}Y
=
\partial_{r}Y_{1}
+
\frac{\partial_{r}c_{1}}
{c_{1}}
Y_{1}
+
\partial_{\theta}Y_{2}.
$$
Thus,
divergence-freeness of $Y$
implies
\begin{eqnarray*}
&&\Omega(Y,[Y,Z])
\\
&=&
	-
	\int_{-d}^{d}
	\int_{-\pi}^{\pi}
	B
	\left(
		c_{1}^{2} Y_{1}^{2} \partial_{r}F
	+
	2c_{1}^{2}Y_{1} F
	\left(
	\partial_{r}Y_{1}
	+
	\frac{\partial_{r}c_{1}}
	{c_{1}}
	Y_{1}
	\right)
	\right)
	drd\theta
	\\
	&=&
	-
	\int_{-d}^{d}
	\int_{-\pi}^{\pi}
	B
	\left(
	c_{1}^{2} Y_{1}^{2} \partial_{r}F
	+
	c_{1}^{2}\partial_{r}(Y_{1}^{2}) F
	+
   \partial_{r}(c_{1}^{2})Y_{1}^{2} F
	\right)
	drd\theta.
	\\
\end{eqnarray*}
By the Stokes theorem,
this is equal to
\begin{equation*}
	=\int_{-d}^{d}\int_{-\pi}^{\pi}
	\partial_{r}B
	c_{1}^{2}
	Y_{1}^{2}
	F
	dr d\theta.
\end{equation*}
This completes the proof.
\end{proof}

\begin{Corollary}
\label{MCMC}
		Let $Z=F(r)\partial_{\theta}$ be a zonal flow.
	Then, we have
	\begin{eqnarray*}
	MC_{(Z,a),(Y,b)}
	=
	MC_{Z,Y}
	-a
	\int_{-d}^{d}\int_{-\pi}^{\pi}
	c_{1}^{2}
	Y_{1}^{2}
	F
	\sqrt{1-\dot{c}_{1}^{2}}
	dr d\theta
	\end{eqnarray*}
for $Y=Y_{1}\partial_{r}+Y_{2}\partial_{\theta}\in {\mathfrak g}$.
\end{Corollary}

\begin{proof}
	This is a consequence of \eqref{MC-CE}, Lemma \ref{omegaZ} and the definition of $c_{1},c_{2}$.
	Note that $\Omega([Z,Y],Y)=\Omega(Y,[Y,Z])$
	by the fact that $\Omega$ is a 2-cocycle.
\end{proof}

\begin{proof}[Proof of Theorem \ref{Main}]
Corollary \ref{MCMC} and \eqref{MC-CE} imply the theorem.
\end{proof}
	
\begin{proof}[Proof of Corollary \ref{>0}]
	This follows from Theorem \ref{Main} and
	Fact \ref{TTthm}.
\end{proof}

\begin{Fact}[{\cite[Thm.\:1.2]{TTconj}}]
	\label{TTthm}
	Let $s>1$. Then,
	for any zonal flow $Z$ on $M_{s}$ whose support is contained in $M_{s}\backslash\{(0,0,\pm 1)\}$,
	there exists $Y\in {\mathfrak g}$
	satisfying $MC_{Z,Y}>0$.
\end{Fact}

\section{Final remark}
Note that we do not know whether $\widehat{\mathcal D}_{\mu}(M_{s})$,
whose existence is guaranteed by Proposition \ref{Existence-D},
satisfies the assumption of Fact \ref{Mcriterion}.
Therefore,
we cannot conclude the existence of a conjugate point on
$\widehat{\mathcal D}_{\mu}(M_{s})$.
This is still under intensive research.

\renewcommand{\thesection}{\Alph{section}}
\setcounter{section}{0}
	
\section{Appendix}

\subsection{Sign conventions}
In this subsection,
we briefly derive the formula \eqref{A-E-eq-D} in order to
clarify our sign conventions.	
Therefore all contents in this section are known.

We refer to \cite[Section 46]{C-G} or \cite[Section 2]{Vizman-GEDG}.

\subsubsection{Right-invariant Maurer-Cartan form}
Let $G$ be a (possibly infinite-dimensional) Lie group and ${\mathfrak g}$ its Lie algebra.
\begin{Definition}
	The right-invariant Maurer-Cartan form $\omega$
	is the ${\mathfrak g}$-valued 1-form on $G$ defined by
	$$
	\omega_{g}(X) :=
	r_{g^{-1}}X \in {\mathfrak g},
	$$
	where $r_{g^{-1}}$ is the differential of the right translation map
	$R_{g^{-1}}(h):=hg^{-1}$.
\end{Definition}
	
	For $X\in {\mathfrak g}$, we write $X^{R}$ for the right-invariant vector fields on $G$ with $X^{R}(e)=X$.
	Note that
	\begin{eqnarray}
	[X^{R},Y^{R}] = - [X,Y]^{R}.
	\label{sign-bracket}
	\end{eqnarray}
	
	\begin{Lemma}
	For $X,Y\in{\mathfrak g}$, we have
		$$
		\omega([X^{R},Y^{R}])
		=
		[\omega(X^{R}),\omega(Y^{R})].
		$$
	\end{Lemma}
	\begin{proof}
	By \eqref{sign-bracket},
	we have
	\begin{eqnarray*}
		\omega([X^{R},Y^{R}])
		=
		\omega(-[X,Y]^{R})
		=
		-[X,Y]^{R}
		=
		[X^{R},Y^{R}]
		=
		[\omega(X^{R}),\omega(Y^{R})],
		\end{eqnarray*}
		which completes the proof.
	\end{proof}
	
	\begin{Lemma}
	For smooth vector fields $U,V$ on $G$,
	we have
		$$
		d\omega(U,V)
		=
		-[\omega(U),\omega(V)].
		$$
	\end{Lemma}
	\begin{proof}
		Recall that
		$$
		d\omega(U,V)
		=
		U(\omega(V))
		-
		V(\omega(U))
		-
		\omega([U,V]).
		$$
		In the case of $U=X^{R}$ and $V=Y^{R}$,
		we have
		\begin{eqnarray*}
		d\omega(X^{R},Y^{R})
		&=&
		-\omega([X^{R},Y^{R}])
		\\
		&=&
		-[\omega(X^{R}),\omega(Y^{R})].
		\end{eqnarray*}
		This equation is the one as ${\mathfrak g}$-valued 2-forms.
		Thus this holds for any $U,V$.
	\end{proof}

	\begin{Corollary}
	\label{Lemma-R-M-UVUV}
	For smooth vector fields $U,V$ on $G$,
	we have
		$$
		[\omega(U),\omega(V)]
		=
		-
		U(\omega(V))
		+
		V(\omega(U)).
		$$
	\end{Corollary}
	\begin{proof}
		This is obvious by preceding lemma.
	\end{proof}

\subsubsection{Euler-Arnol'd equation}
Let $G$ be a (possibly infinite-dimensional) Lie group with right-invariant metric $\langle,\rangle$
and
${\mathfrak g}=T_{e}G$ the Lie algebra.
Define $[,]^{*}$ by
\begin{eqnarray*}
	\langle [X,Y]^{*},Z\rangle
	=
	\langle Y,[X,Z]\rangle
\end{eqnarray*}
if it exists.
We always assume the existence of $[,]^{*}$ in this article.

\begin{Lemma}
	Let $\eta$ be a curve on $G$.
	Define a curve $c:[0,t_{0}] \to {\mathfrak g}$ by
	$c(t):=r_{\eta^{-1}}(\dot{\eta})$.
	Then, $\eta$ is a geodesic if and only if $c$ satisfies
	\begin{eqnarray}
	\partial_{t}c=[c,c]^{*}.
	\label{L-eq-EA}
	\end{eqnarray}
	Moreover, for a curve $c:[0,t_{0}] \to {\mathfrak g}$ satisfying \eqref{L-eq-EA},
	there exists 
	a geodesic $\eta$ on $G$ satisfying $c(t)=r_{\eta^{-1}}(\dot{\eta})$ if $G$ is regular in the sense of \cite[Def.\:7.6]{Milnor}.
\end{Lemma}

\begin{proof}
Consider the energy function of a curve $\eta$ on $G$: 
\begin{eqnarray*}
	E(\eta) 
	&=& 
	\frac{1}{2}\int_{0}^{t} 
	\langle \dot{\eta},\dot{\eta}\rangle dt\\
	&=&
	\frac{1}{2}\int_{0}^{t}
	|| r_{\eta^{-1}}(\dot{\eta})|| dt.
	\end{eqnarray*}
	For a proper variation $\eta_{s}$ of $\eta$,
	define $c_{s}(t) := r_{\eta^{-1}}(\dot{\eta})
	=\omega(\dot{\eta})\in {\mathfrak g}$,
	$X_{s}(t) := \partial_{s}\gamma_{s}(t)$,
	and $x_{s}(t) := \omega(X_{s})$
	where
	$\omega$ is the right-invariant Maurer-Cartan form.
	Then,
the first variation is
\begin{eqnarray*}
	\partial_{s}E(\eta_{s})
	&=&
	\int_{0}^{t}
	\langle c_{s},\partial_{s}c_{s}\rangle dt.
\end{eqnarray*}
Corollary \ref{Lemma-R-M-UVUV} implies
 \begin{eqnarray*}
	\partial_{s}c_{s} &=& X_{s}(\omega(\dot{\eta}))
	\\
	&=& 
	[\omega(\dot{\eta}),\omega(X_{s})]
	+
	\dot{\eta}(\omega(X_{s}))
	\\
	&=&
	[c_{s},x_{s}]+\partial_{t}x_{s}.
\end{eqnarray*}
	Thus,
	\begin{eqnarray*}
		\partial_{s}E(\eta_{s})
		&=&
		\int_{0}^{t}
		\langle c_{s}, [c_{s},x_{s}]+\partial_{t}x_{s}\rangle dt
		\\
		&=&
		\int_{0}^{t}
		\langle [c_{s},c_{s}]^{*}-\partial_{t}c_{s}, x_{s}\rangle dt.
	\end{eqnarray*}
	This completes the proof.
\end{proof}

	\begin{Definition}
		We define the Euler-Arnol'd equation of $G$ by
	\begin{eqnarray*}
		u_{t} = [u,u]^{*}.
	\end{eqnarray*}
	for $u:[0,t_{0}]\to {\mathfrak g}$. In other words 
	\begin{eqnarray*}
	u_{t} = -\operatorname{ad}_{u}^{*}u,
	\end{eqnarray*}
	where $\operatorname{ad}_{u}v=-[u,v]$ and 
	$\langle \operatorname{ad}_{u}v,w\rangle
	=
	\langle v,\operatorname{ad}_{u}^{*}w\rangle$.
	\end{Definition}

	\subsubsection{Euler-Arnol'd equation of ${\mathcal D}_{\mu}(M)$}

	Let  ${\mathcal D}_{\mu}(M)$ be the group of volume-preserving $C^{\infty}$-diffeomorphisms of a $n$-dimensional compact Riemannian manifold $(M,g)$ with right-invariant Riemannian metric
$$
\langle X,Y\rangle:=\int_{M} g(X,Y)\mu,
$$
where $X,Y\in{\mathfrak g}:=T_{e}{\mathcal D}_{\mu}(M)$.
Let $P:{\mathfrak X}(M)\to {\mathfrak g}$ be the projection to the divergence-free part, where ${\mathfrak X}(M)$ is the space of vector fields.

\begin{Lemma}
	For $X,Y\in{\mathfrak g}$, we have
	$$
	\nabla_{X^{R}}Y^{R} = - (P(\nabla^{M}_{X}Y))^{R},
	$$
	where $\nabla^{M}$ is the Levi-Civita connection on $M$.
	\begin{proof}
	The Koszul formula and the right-invariance imply
		\begin{eqnarray*}
			2\langle \nabla_{X^{R}}Y^{R},Z^{R}\rangle
			&=&
			\langle [X^{R},Y^{R}], Z^{R}\rangle
			-
			\langle [X^{R},Z^{R}], Y^{R}\rangle
			-
			\langle [Y^{R},Z^{R}], X^{R}\rangle
			\\
			&=&
			-\langle [X,Y]^{R}, Z^{R}\rangle
			+
			\langle [X,Z]^{R}, Y^{R}\rangle
			+
			\langle [Y,Z]^{R}, X^{R}\rangle
			\\
			&=&
			-\langle [X,Y], Z\rangle
			+
			\langle [X,Z], Y\rangle
			+
			\langle [Y,Z], X\rangle
		\end{eqnarray*}
		This completes the proof by
		the Koszul formula on $M$
		and the right-invariance of $\langle,\rangle$.
		\end{proof}
\end{Lemma}

\begin{Lemma}
For $X\in{\mathfrak g}$, we have
	$$
	[X,X]^{*}=-P(\nabla^{M}_{X}X).
	$$
\end{Lemma} 
	
\begin{proof}

By the Koszul formula, we have
\begin{eqnarray*}
	2\langle \nabla^{M}_{X}X,Z\rangle 
	&=&
	-2
	\langle \nabla_{X^{R}}X^{R},Z^{R}\rangle
	\\
	&=&
	-
	\langle [X^{R},X^{R}],Z^{R}\rangle
	+
	\langle [X^{R},Z^{R}],X^{R}\rangle
	+
	\langle [X^{R},Z^{R}],X^{R}\rangle
	\\
	&=&
	2\langle X^{R},[X^{R},Z^{R}]\rangle
	\\
	&=&
	-2
	\langle X^{R},[X,Z]^{R}\rangle
	\\
	&=&
	-2\langle X,[X,Z]\rangle.
\end{eqnarray*}
This completes the proof.
\end{proof}

\begin{Corollary}
The Euler-Arnol'd equation of ${\mathcal D}_{\mu}(M)$ is
\begin{eqnarray*}
	u_{t} = -\nabla^{M}_{u}u - \operatorname{grad}p.
\end{eqnarray*}
\end{Corollary}

\subsubsection{Lichnerowicz 2-cocycle}
Let ${\mathcal D}_{\mu}(M)$ be the groups of volume-preserving diffeomorphisms of a compact $n$-dimensional Riemannian manifold $M$,
and ${\mathfrak g}:=T_{e}{\mathcal D}_{\mu}(M)$,
which is identified with the space of divergence-free vector fields on $M$.
	For a closed 2-form $\eta$, define a skew-symmetric bilinear form $\Omega$ on the Lie algebra ${\mathfrak g}$ by
	$$
	\Omega(X,Y):=\int_{M}\eta(X,Y).
	$$
	\begin{Lemma}
		The form $\Omega$ defines a 2-cocycle on ${\mathfrak g}$,
		namely,
		it satisfies the ``Jacobi identity''
		\begin{eqnarray*}
			\Omega([X,Y],Z)+
		\Omega([Z,X],Y)+
		\Omega([Y,Z],X)
		=0
		\end{eqnarray*}
		for any $X,Y,Z\in{\mathfrak g}$.
	\end{Lemma}
	\begin{proof}
		Because $\eta$ is closed,
there exists a one-form $\alpha$ on $M$ such that
$\eta=d\alpha$.
Recall the formula of the exterior derivative of 1-form:
\begin{equation}
	d\alpha(X,Y)
	=
	X(\alpha(Y))
	-
	Y(\alpha(X))
	-
	\alpha([X,Y]).
\end{equation}
This implies
	\begin{eqnarray*}
		\Omega([X,Y],Z)
		&=&
		\int_{M}
		\left(
		[X,Y](\alpha(Z))
	-
	Z(\alpha([X,Y]))
	-
	\alpha([[X,Y],Z])
		\right)\mu
		\\
		&=&
		\int_{M}
		\alpha([[X,Y],Z])
		\mu.
	\end{eqnarray*}
	Note that the second equality follows from  
	$$
	\int_{M}X(f)\mu =
	\int_{M}g(\operatorname{grad}f,X)\mu
	=0
	$$
	for any $f\in C^{\infty}(M)$ and a divergence-free vector field $X$.
	Thus,
	we have
\begin{eqnarray*}
	&&\Omega([X,Y],Z)+
		\Omega([Z,X],Y)+
		\Omega([Y,Z],X)
		\\
	&=&
	\int_{M}
	\left(
		\alpha([[X,Y],Z]
		+
		[[Z,X],Y]
		+
		[[Y,Z],X])
		\right)
		\mu
		\\
		&=&
		0
\end{eqnarray*}
by the Jacobi identity of $[,]$.
This completes the proof.
\end{proof}

	This Lemma allows us to
	endow $\widehat{\mathfrak g}:={\mathfrak g}\oplus {\mathbb R}$
	with a Lie algebra structure defined by
	$$
	[(X,a),(Y,b)]:=([X,Y],\Omega(X,Y)).
	$$
	Take $B\in C^{\infty}(\wedge^{n-2}TM)$ satisfying 
	$\iota_{B}\mu =\eta$.
	Then, we have
	$$
	\Omega(X,Y) 
	=
	\int_{M}\iota_{B}\mu(X,Y) 
	=
	\int_{M}\mu(B,X,Y)
	=
	\int_{M} g(B\times X,Y)
	=
	\langle P(B\times X),Y\rangle,
	$$
	where $B\times X =\star (B\wedge X)$
	(see \eqref{def-cross} and Lemma \ref{Lemma-cross-product})
	and
	$P:{\mathfrak X}(M)\to {\mathfrak g}$ is the projection to the divergence-free part.
	 
	\begin{Lemma}
	For $(X,a),(Y,b)\in\widehat{\mathfrak g}$, we have
		$$
		[(X,a),(Y,b)]^{*}
		=
		([X,Y]^{*}+bP(B\times X),0).
		$$
	\end{Lemma}
	\begin{proof}
		\begin{eqnarray*}
			\langle [(X,a),(Y,b)]^{*},(Z,c)\rangle
			&=&
			\langle (Y,b),[(X,a),(Z,c)]\rangle
			\\
			&=&
			\langle (Y,b),([X,Z],\Omega(X,Z)\rangle
			\\
			&=&
			\langle Y,[X,Z]\rangle + b \Omega(X,Z)
			\\
			&=&
			\langle [X,Y]^{*},Z\rangle +\langle bP(B\times X),Z\rangle.
		\end{eqnarray*}
		This completes the proof.
		\end{proof}

	Thus, we have 
	\begin{Theorem}
		The Euler-Arnol'd equation of $\widehat{\mathfrak g}$
		is
		\begin{eqnarray*}
			u_{t} = [u,u]^{*}+aP(B\times u),
		\end{eqnarray*}
		or equivalently,
		\begin{eqnarray*}
			u_{t} = -\nabla_{u}u+a(B\times u)-\operatorname{grad}p.
		\end{eqnarray*}
	\end{Theorem}
	
	\subsubsection{Formulae on Riemannian manifold}

In this subsection, for the convenience of readers, we briefly summarize formulae on Riemannian manifold.
	Let $(M,g)$ be a $n$-dimensional Riemannian manifold
and
$\mu$ the volume form.
Write ${\mathfrak X}^{p}(M)$ for the space of $p$-vector fields on $M$,
and
${\mathcal E}^{q}(M)$ for the space of 
$q$-forms on $M$.

\begin{Definition}
\label{Def-musical}
	Define 
	$\flat: {\mathfrak X}^{1}(M)\to {\mathcal E}^{1}(M)$
	and
	$\sharp:{\mathcal E}^{1}(M)\to {\mathfrak X}^{1}(M)$
	by
	\begin{alignat*}{2}
		V^{\flat}:=&
		g(V,\cdot)
		\quad &&\in {\mathcal E}^{1}(M),
		\\
		g(\alpha^{\sharp},\cdot)=&
		\alpha
		&&\in {\mathcal E}^{1}(M)
	\end{alignat*}
	for $V\in{\mathfrak X}^{1}(M)$
	and
	$\alpha\in{\mathcal E}^{1}(M)$.
	We extend these isomorphisms to
	$\flat: {\mathfrak X}^{p}(M)\to {\mathcal E}^{p}(M)$
	and
	$\sharp:{\mathcal E}^{p}(M)\to {\mathfrak X}^{p}(M)$
	for any $p\in{\mathbb Z}$.
\end{Definition}

\begin{Definition}
\label{langle-rangle}
	Define 
	$\langle,\rangle_{\mathfrak X}:
	{\mathfrak X}^{p}(M)\otimes_{C^{\infty}(M)}
	{\mathfrak X}^{p}(M)
	\to 
	C^{\infty}(M)$
	by
	\begin{equation*}
		\langle V,
		W\rangle_{\mathfrak X}
		:=
		\iota_{V}(W^{\flat}),
	\end{equation*}
	where
	$\iota$ is the interior derivative.
	Similarly,
	define 
	$\langle,\rangle_{\mathcal E}:
	{\mathcal E}^{p}(M)\otimes_{C^{\infty}(M)}
	{\mathcal E}^{p}(M)
	\to 
	C^{\infty}(M)$
	by
	\begin{equation*}
		\langle\alpha,
		\beta\rangle_{\mathcal E}
		:=
		\iota_{\alpha^{\sharp}}
		(\beta).
	\end{equation*}
\end{Definition}

\begin{Lemma}
	Let $V,W\in{\mathfrak X}^{1}(M)$.
	Then,
	we have
	\begin{equation*}
		\langle V,W
		\rangle_{\mathfrak X}
		=
		g(V,W).
	\end{equation*}
\end{Lemma}

\begin{proof}
By Definition \ref{Def-musical}
and \ref{langle-rangle},
we have
	\begin{equation*}
		\langle V,W
		\rangle_{\mathfrak X}
		=
		\iota_{V}(W^{\flat})
		=
		W^{\flat}(V)
		=
		g(W,V).
	\end{equation*}
	This completes the proof.
\end{proof}

\begin{Definition}
\label{Def-Hodge}
	We define the Hodge star operator
	$\star:{\mathfrak X}^{p}(M)
	\to
	{\mathfrak X}^{n-p}(M)$
	and
	$\star:{\mathcal E}^{p}(M)
	\to
	{\mathcal E}^{n-p}(M)$
	by
	\begin{alignat}{1}
	V\wedge \star W=&
		\langle V,W\rangle_{\mathfrak X}  \mu^{\sharp}
		\qquad
		\text{
		for any
		}
		V\in{\mathfrak X}^{p}(M)
		\label{formula-Hodge-vect}
		\\
		\alpha\wedge (\star \beta)
		=&
		\langle \alpha,\beta\rangle_{\mathcal E} \mu
		\qquad
		\text{
		for any
		}
		\alpha\in{\mathcal E}^{p}(M).
	\end{alignat}
\end{Definition}
Note that
\begin{equation}
\star^{2}\alpha=(-1)^{n-1}\alpha
\quad
\text{
for
}
\alpha
\in{\mathcal E}^{1}(M).
\label{Hodge=-1}
\end{equation}
Moreover,
applying $\mu$ to \eqref{formula-Hodge-vect},
we have
\begin{equation}
	\mu(V\wedge\star W)
	=\langle V,W\rangle_{\mathfrak X}.
	\label{formula-Hodge-mu}
\end{equation}

\begin{Definition}
\label{Definition-cross-product}
	For
	$X\in{\mathfrak X}^{1}(M)$
	and
	$B\in{\mathfrak X}^{n-2}(M)$,
	define
	\begin{equation}
		B\times X
		:=
		\star(B\wedge X)
		\quad
		\in{\mathfrak X}^{1}(M).
		\label{def-cross}
	\end{equation}
\end{Definition}

\begin{Lemma}
\label{Lemma-cross-product}
	For $X,Y\in{\mathfrak X}^{1}(M)$
	and
	$B\in{\mathfrak X}^{n-2}(M)$,
	we have
	\begin{equation*}
		\int_{M}
		\iota_{B}(\mu)(X,Y)\mu
		=
		\int_{M}
		g(B\times X,Y)\mu.
	\end{equation*}
\end{Lemma}

\begin{proof}
By the definition,
we have
	\begin{alignat*}{1}
		\int_{M}
		\iota_{B}(\mu)(X,Y)\mu
		=&
		\int_{M}
		\mu(B,X,Y)\mu.
		\end{alignat*}
On the other hand,
    \begin{alignat*}{1}
    	\int_{M}
    	g(B\times X,Y)\mu
      	=&
    	\int_{M}
    	\langle
    	Y,
    	B\times X
    	\rangle
    	\mu
    	\\
    	=&
    	\int_{M}
    	\langle
    	Y,
    	\star(B\wedge X)
    	\rangle
    	\mu
    	\\
    	=&
    	\int_{M}
    	\mu(
    	Y\wedge
    	\star^{2}(B\wedge X)
    	)
    	\mu
    	\\
    	=&
    	(-1)^{n-1}\int_{M}
		\mu(Y\wedge B\wedge X)\mu\\
		=&
    	\int_{M}
		\mu(B,X,Y)\mu.
    \end{alignat*}
    The third and fifth equalities follow from
    \eqref{Hodge=-1}    
    and
    \eqref{formula-Hodge-mu},
    respectively.
\end{proof}

\section{Appendix II}
\subsection{Misio{\l}ek curvature}
In this section,
we briefly recall the Misio{\l}ek curvature.
We refer to \cite{Misiolek-T,TTconj}.

\subsubsection{Curvature on group with right-invariant metric}
In this subsection,
we recall the formulae concerning to a group with right-invariant metric.
Thus, all contents in this subsection are known.

Let $G$ be a (possibly infinite-dimensional) Lie group with a right-invariant metric $\langle,\rangle$, and
${\mathfrak g}$ its Lie algebra.
Define
$$
\langle [X,Y],Z\rangle
=
\langle Y,[X,Z]^{*}\rangle.
$$

\begin{Lemma}
	For $X,Y\in{\mathfrak g}$,
	we have
	\begin{eqnarray*}
		2\nabla_{X^{R}}Y^{R}
		=
		(-[X,Y]+[X,Y]^{*}+[Y,X]^{*})^{R}
	\end{eqnarray*}
	where $\nabla$ is the right-invariant Levi-Civita connection on $G$. In particular,
	$$
	\nabla_{X^{R}}X^{R} = [X,X]^{*R}.
	$$
\end{Lemma}
\begin{proof}
	The Koszul formula and the right-invariance imply
	\begin{eqnarray*}
		2\langle \nabla_{X^{R}}Y^{R},Z^{R}\rangle
		&=&
		\langle [X^{R},Y^{R}],Z^{R}\rangle
		-
		\langle [X^{R},Z^{R}],Y^{R} \rangle
		-
		\langle [Y^{R},Z^{R}],X^{R}\rangle
		\\
		&=&
		-
		\langle [X,Y]^{R},Z^{R}\rangle
		+
		\langle [X,Z]^{R},Y^{R} \rangle
		+
		\langle [Y,Z]^{R},X^{R}\rangle
		\\
		&=&
		-
		\langle [X,Y],Z\rangle
		+
		\langle [X,Z],Y \rangle
		+
		\langle [Y,Z],X\rangle
		\\
		&=&
		\langle -[X,Y]+[X,Y]^{*}+[Y,X]^{*},Z\rangle.
	\end{eqnarray*}
	This completes the proof.
\end{proof}
For the simplicity,
put 
$$
A^{\pm}(X,Y)
=
[X,Y]^{*} \pm [Y,X]^{*}.
$$
Define
$$
R(X^{R},Y^{R})
=
\nabla_{X^{R}}\nabla_{Y^{R}}
-
\nabla_{Y^{R}}\nabla_{X^{R}}
-
\nabla_{[X^{R},Y^{R}]}.
$$
\begin{Lemma}[{\cite[Thm. 2.1 in IV. \S2]{AK}}]
\label{curvature-R-inv}
	For $X,Y\in{\mathfrak g}$,
	we have
	\begin{eqnarray*}
		&&4\langle R(X^{R},Y^{R})Y^{R},X^{R}\rangle \\
		=&&
		-4\langle [Y,Y]^{*},[X,X]^{*}\rangle
		+||A^{+}(X,Y)||^{2}
		-3||[X,Y]||^{2}
		-2\langle [X,Y],A^{-}(X,Y)\rangle.
	\end{eqnarray*}
\end{Lemma}
\begin{proof}
	By the property of the Levi-Civita connection and the right-invariance,
	we have
	\begin{eqnarray*}
		4\langle\nabla_{X^{R}}\nabla_{Y^{R}}Y^{R},X^{R}\rangle
		&=&
		4X^{R}
		\langle \nabla_{Y^{R}}Y^{R},X^{R}\rangle
		-4
		\langle \nabla_{Y^{R}}Y^{R},\nabla_{X^{R}}X^{R}\rangle
		\\
		&=&
		-4\langle [Y,Y]^{*R},[X,X]^{*R}\rangle.
	\end{eqnarray*}
	Similarly,
	\begin{eqnarray*}
		-4\langle\nabla_{Y^{R}}\nabla_{X^{R}}Y^{R},X^{R}\rangle
		&=&
		-4Y^{R}
		\langle \nabla_{X^{R}}Y^{R},X^{R}\rangle
		+4
		\langle \nabla_{X^{R}}Y^{R},\nabla_{Y^{R}}X^{R}\rangle
		\\
		&=&
		\langle(-[X,Y]+A^{+}(X,Y))^{R} ,
		(-[Y,X]+A^{+}(X,Y))^{R}\rangle
		\\
		&=&
		||A^{+}(X,Y)||^{2} - ||[X,Y]||^{2}.
	\end{eqnarray*}
	Moreover,
	\begin{eqnarray*}
		-4\langle \nabla_{[X^{R},Y^{R}]} Y^{R},X^{R}\rangle
		&=&
		4\langle \nabla_{[X,Y]^{R}} Y^{R},X^{R}\rangle
		\\
		&=&
		2\langle
		(-[[X,Y],Y]+[[X,Y],Y]^{*}+[Y,[X,Y]]^{*},X\rangle
		\\
		&=& 
		2\langle 
		[X,Y],[Y,X]^{*}
		\rangle
		+2\langle 
		Y, [[X,Y],X]
		\rangle
		+2\langle
		[X,Y],[Y,X]\rangle
		\\
		&=&
		2\langle [X,Y], [Y,X]^{*}-[X,Y]^{*}\rangle
		-2||[X,Y]||^{2}
		\\
		&=&
		-2\langle [X,Y],A^{-}(X,Y)\rangle
		-2||[X,Y]||^{2}.		
	\end{eqnarray*}
	This completes the proof.
\end{proof}

\subsubsection{Definition of Misio{\l}ek curvature}

Let $G$ be a (possibly infinite-dimensional) Lie group with a right-invariant metric $\langle,\rangle$,
${\mathfrak g}$ its Lie algebra.
Define
\begin{eqnarray*}
	\langle [X,Y],Z\rangle
	&=&
	\langle Y,[X,Z]^{*}\rangle,
	\\
	R(X^{R},Y^{R})
	&=&
	\nabla_{X^{R}}\nabla_{Y^{R}}
	-
	\nabla_{Y^{R}}\nabla_{X^{R}}
	-
	\nabla_{[X^{R},Y^{R}]},
	\\
	A^{\pm} (X,Y)
	&=&
	[X,Y]^{*} \pm [Y,X]^{*}.
\end{eqnarray*}
Then we have
$$
2\nabla_{X^{R}}Y^{R}=-[X,Y]^{R}+A^{+}(X,Y)^{R}.
$$
Note that $[X,X]^{*}=0$ implies that $X$ is a stationary solution of the Euler-Arnol'd equation of $G$.

\begin{Lemma}
\label{Misiolek-calculation}
	Let $X\in{\mathfrak g}$ satisfying $[X,X]^{*}=0$
	and $\eta$ a geodesic corresponding to $X\in T_{e}G$.
	For $Y\in{\mathfrak g}$ and $f\in C^{\infty}(G)$ with $f(e)=f(t_{0})=0$ for some $t_{0}>0$,
	define a vector field $\widetilde{Y}$ along $\eta$ by
	$
	\widetilde{Y}(\eta(t)) =f(t)\cdot Y^{R}(\eta(t)).
	$
	Then the second variation $E''$ of the energy function of $\eta$ 
	is
	\begin{eqnarray*}
	E''(\eta)(\widetilde{Y},\widetilde{Y})
	=
	\int_{0}^{t_{0}}
	\left(
	\dot{f}^{2}||Y||^{2}
	-
	f^{2}
	\left(
	\langle R(X^{R},Y^{R})Y^{R},X^{R}\rangle
	-
	||\nabla_{X^{R}}Y^{R}||^{2}
	\right)
	\right)
	dt
	\end{eqnarray*}
	where $\dot{f}:=X^{R}f$.
\end{Lemma}
\begin{proof}
	Note that $\dot{\eta}=X^{R}$ by the assumption.
	The general formula for the second variation of the energy function implies
	$$
	E''(\eta)(\widetilde{Y},\widetilde{Y})
	=
	\int_{0}^{t_{0}}
	||
	\nabla_{X^{R}}\widetilde{Y}||^{2}dt
	-
	\int_{0}^{t_{0}}
	\langle
	R(X^{R},\widetilde{Y})\widetilde{Y},X^{R}
	\rangle
	dt.
	$$
	For the first term,
	we have
	\begin{eqnarray*}
		\nabla_{X^{R}}\widetilde{Y}
		&=&
		\dot{f} Y^{R} + f\nabla_{X^{R}}Y^{R}
	\end{eqnarray*}
	Thus,
	\begin{eqnarray*}
		||\nabla_{X^{R}}\widetilde{Y}||^{2}
		=
		\dot{f}^{2}||Y||+2f\dot{f}\langle Y,\nabla_{X^{R}}Y^{R}\rangle +f^{2}||\nabla_{X^{R}}Y^{R}||^{2}.
	\end{eqnarray*}
	Then, $\langle Y^{R},\nabla_{X^{R}}Y^{R}\rangle=-\langle \nabla_{X^{R}}Y^{R},Y^{R}\rangle$
	implies
	\begin{eqnarray*}
		||\nabla_{X^{R}}\widetilde{Y}||^{2}
		&=&
		\dot{f}^{2}||Y||^{2}
		+f^{2}||\nabla_{X^{R}}Y^{R}||^{2}
	\end{eqnarray*}
	On the other hand, 
	$$
	\langle R(X^{R},\widetilde{Y})\widetilde{Y},X^{R}\rangle
	=
	f^{2}
	\langle R(X^{R},Y^{R})Y^{R},X^{R}\rangle.
	$$
	This completes the proof.
\end{proof}

\begin{Definition}
\label{Appx-MC-def}
	Let $X\in{\mathfrak g}$ satisfying $[X,X]^{*}=0$.
	Define the Misio{\l}ek curvature $MC_{X,Y}:=MC^{G}_{X,Y}$ by
	$$
	MC_{X,Y} = 
	\langle R(X^{R},Y^{R})Y^{R},X^{R}\rangle
	-
	||\nabla_{X^{R}}Y^{R}||^{2}
	$$
	for $Y\in{\mathfrak g}$.
\end{Definition}

\begin{Theorem}
	Let $X\in{\mathfrak g}$ satisfying $[X,X]^{*}=0$
	and $\eta$ a geodesic corresponding to $X\in T_{e}G$.
	Suppose $Y\in{\mathfrak g}$ satisfies $MC_{X,Y}>0$.
	For $s>0$, define
	\begin{eqnarray*}
	t_{s}&:=&\pi ||Y||\sqrt{ \frac{s}{MC_{X,Y}}} \qquad\in {\mathbb R}_{>0}
	\\
	f_{s}(t)&:=&
	\sin\left(
	\frac{t}{||Y||}
	\sqrt{
	\frac{MC_{X,Y}}{s}
	}
	\right)
	\qquad \in C^{\infty}({\mathbb R}_{\geq 0})
	\end{eqnarray*}
	and a vector field $\widetilde{Y}$ along $\eta$
	by $\widetilde{Y}(\eta(t))=f_{s}(t)Y^{R}(\eta(t))$.
	Then, the second variation $E''$ of the energy function of $\eta$ 
	is
	\begin{eqnarray*}
	E''(\eta)(\widetilde{Y},\widetilde{Y})
	=
	\frac{\pi}{2}(1-s)||Y||\sqrt{\frac{MC_{X,Y}}{s}}.
	\end{eqnarray*}
	In particular, 
	$E''(\eta)(\widetilde{Y},\widetilde{Y})<0$ if $0<s<1$.
\end{Theorem}
\begin{proof}
For simplicity, set $M:=MC_{X,Y}$.
	By Lemma \ref{Misiolek-calculation},
	we have
	\begin{eqnarray*}
		&&E''(\eta)(\widetilde{Y},\widetilde{Y})
		\\
		=&&\int_{0}^{t_{s}}\left(
		\dot{f}_{s}^{2}||Y||^{2}
		-f_{s}^{2}M\right)dt
		\\
		=&&
		\int_{0}^{t_{s}}
		\left(
		\frac{M}{s}
		\cos^{2}\left(\frac{t}{||Y||}
	\sqrt{
	\frac{M}{s}}\right)
	-
	M
	\sin^{2}\left(
	\frac{t}{||Y||}
	\sqrt{
	\frac{M}{s}
	}
	\right)
	\right)dt
	\\
	=&&
	M\int_{0}^{\pi}
	\left(
	\frac{1}{s}\cos^{2}(x)
	-
	\sin^{2}(x)
	\right)||Y||\sqrt{\frac{s}{M}}dx
	\\
	=&&
	\frac{\pi}{2}(1-s)||Y||\sqrt{\frac{M}{s}}.
	\end{eqnarray*}
This completes the proof.
\end{proof}

\begin{Lemma}
\label{MC-other}
	Let $X\in{\mathfrak g}$ satisfying $[X,X]^{*}=0$
	and $Y\in{\mathfrak g}$.
	Then,
	we have
	$$
	MC_{X,Y}
	=
	-
	||[X,Y]||^{2}
	-
	\langle
	[[X,Y],Y],X\rangle.
	$$
\end{Lemma}

\begin{proof}
By Lemma \ref{curvature-R-inv},
we have
\begin{eqnarray*}
	4\langle R(X^{R},Y^{R})Y^{R},X^{R}\rangle 
	&=&
		||A^{+}(X,Y)||^{2}
		-3||[X,Y]||^{2}
		-2\langle [X,Y],A^{-}(X,Y)\rangle.
\end{eqnarray*}	
On the other hand,
\begin{eqnarray*}
	-4||\nabla_{X^{R}}Y^{R}||^{2}
	&=&
	-
	||
	-[X,Y]+A^{+}(X,Y)
	||
	\\
	&=&
	-||[X,Y]||^{2}
	+2
	\langle [X,Y], A^{+}(X,Y)\rangle
	-
	||A^{+}(X,Y)||^{2}.
\end{eqnarray*}
This completes the proof.
\end{proof}

\subsubsection{Misio{\l}ek curvature of ${\mathcal D}_{\mu}(M)$}
Let ${\mathcal D}_{\mu}(M)$ be the group of volume-preserving $C^{\infty}$-diffeomorphisms of a compact $n$-dimensional manifold $M$ 
with the $L^{2}$ right-invariant metric:
\begin{eqnarray*}
	\langle X,Y\rangle :=
	\int_{M}g(X,Y)\mu.
\end{eqnarray*}
Here $X,Y\in{\mathfrak g}=T_{e}G$,
which is identified with the space of divergence-free vector fields.
\begin{Lemma}
	Let $X\in{\mathfrak g}$ satisfying $[X,X]^{*}=0$ and $Y\in{\mathfrak g}$.
	Then,
	we have
	$$
	MC_{X,Y}=\langle \nabla^{M}_{X}[X,Y]+\nabla^{M}_{[X,Y]}X,Y\rangle,
	$$
	where $\nabla^{M}$ is the Levi-Civita connection on $M$.
\end{Lemma}

\begin{proof}
	By the Koszul formula,
	we have
	\begin{eqnarray*}
		2\langle \nabla^{M}_{X}[X,Y], Y\rangle
		&=&
		\langle [X,[X,Y]], Y\rangle
		-
		\langle [X,Y], [X,Y]\rangle
		-
		\langle [[X,Y],Y],X\rangle,
		\\
		2\langle \nabla^{M}_{[X,Y]}X, Y\rangle
		&=&
		\langle [[X,Y],X], Y\rangle
		-
		\langle [[X,Y],Y],X\rangle
		-
		\langle [X,Y],[X,Y]\rangle.
	\end{eqnarray*}
	Thus Lemma \ref{MC-other} implies the lemma.
\end{proof}

\end{document}